\documentclass[11pt]{amsart}
\usepackage[T1]{fontenc}
\usepackage[utf8]{inputenc}
\usepackage[english]{babel}
\usepackage{amsmath, amsthm, amssymb, amsfonts, enumerate, mathrsfs, dsfont, comment, bm, mathtools}
\usepackage{enumitem}
\usepackage[colorlinks=true,linkcolor=blue,urlcolor=blue]{hyperref}
\usepackage{color, geometry, todonotes, indentfirst, graphicx, pdfpages}
\usepackage[title]{appendix}
\usepackage{fancyhdr}
\usepackage{polski}

\newtheorem{theorem}{Theorem}[section]
\newtheorem{remark}[theorem]{Remark}
\newtheorem{assumption}[theorem]{Assumption}
\newtheorem{lemma}[theorem]{Lemma}
\newtheorem{proposition}[theorem]{Proposition}

\newtheorem{definition}[theorem]{Definition}

\theoremstyle{plain}
\def \R{\mathbb{R}}
\def \N{\mathbb{N}}

\def \P{\mathbb{P}}
\def \E{{\mathbb{E}}}
\def \F{\mathbb{F}}

\def \cF{{\mathcal F}}

\def \A{\mathcal{A}}

\def \and{\quad \text{and} \quad}
\def \b{\bar}

\DeclareMathOperator*{\argmin}{arg\,min}

\DeclareMathOperator*{\Minimize}{Minimize}

\title[LQS games and applications]{
Linear-quadratic-singular stochastic differential games and applications}
\author[Dianetti]{Jodi Dianetti} 
\keywords{}

\date{\today}

\numberwithin{equation}{section}

\begin{document}
\maketitle 
 
\begin{abstract}
We consider a class of non-cooperative \emph{N}-player non-zero-sum stochastic differential games with singular controls, in which each player can affect a linear stochastic differential equation in order to minimize a cost functional which is quadratic in the state and linear in the control. 
We call these games \emph{linear-quadratic-singular stochastic differential games}. 
Under natural assumptions, we show the existence of open-loop Nash equilibria, which are characterized through a linear system of forward-backward stochastic differential equations.
The proof is based on an approximation via a sequence of games in which players are restricted to play Lipschitz continuous strategies.
We then discuss an application of these results to a  model of capacity expansion in oligopoly markets.
\end{abstract}

\smallskip
{\textbf{Keywords}}: Singular stochastic control, linear quadratic games, stochastic maximum principle, Nash equilibrium

\smallskip 
{\textbf{AMS subject classification}}: 
91A15, 
49N70, 
93E20, 
60H10 

\section{Introduction}
\subsection{Motivation}
Linear-quadratic stochastic differential games (LQ games, in short) are games in which players are allowed to affect a linear stochastic differential equation with an additive control, with the aim of minimizing a cost functional which is quadratic both in the state and in the control. 
For example, we can consider the game in which, for $i=1,...,N$, player $i$ can choose a control $\alpha^i$ and faces the control problem:
\begin{align*}
 \text{(LQ game)} \ 
 & \Minimize\limits_{{\alpha^i}} \  \mathbb{E} \bigg[ \int_0^T \big( X_t Q^i_t X_t + c_t^i (\alpha_t^i)^2 \big) dt + X_T Q^i_T X_T \bigg], \\
& \text{subject to } dX_t^{j}= ( a^j_t + b^j_t X_t^{j} + \alpha^j_t)dt + \sigma_t^j dW_{t}^j, \ X_{0}^{j}=x_0^j, \ j=1,...,N.
\end{align*} 
Here, $X=(X^1,...,X^N)$ denotes the vector of states of the players, which is affected by the Brownian motions $(W^j)_j$, while, for $y\in \R^N$, $yQ_t^i y$ denotes the product $ \sum_{k,j} Q_t^{k,j;i} y^k y^j$.

LQ games have received a huge attention in the literature. 
On the one hand, this class of models represents a precious and rare example of explicitly solvable games (see \cite{Carmona2016}), on the other hand, it serves as a benchmark theoretical tool for describing systems of interacting agents in many applications, ranging from economics, engineering, finance, biology and so on (see \cite{Carmona2016} and the references therein).

However, despite the popularity of LQ games, assuming  the cost  to be quadratic in the control is rather unrealistic in many applications, where the price of an intervention would rather be linear in its size. 
As an example, we can mention models of capacity expansion in oligopoly markets. 
Consider $N$ companies producing a certain good and selling it in the market.
Each company can adjust its production capacity to follow the market fluctuation of the demand, in order to maximize a net profit.
Such a reward is given by the profit obtained by the selling, which corresponds to the individual production multiplied by the market price (which is affected by the production of all the firms in the market) and by the cost of adjusting the production, that we may think as the actual cost of an investment. 
Thus, in this case it is reasonable to assume  the cost of an investment to be linear in the capacity expansion
 (see \cite{Back&Paulsen, grenadier2002option, Steg}).
Other examples in which the costs of intervention is not quadratic arise in resource allocation problems (see \cite{gao.lu.sharma.squillante.bosman2018bounded, georgiadis.needly.tassiulas.2006resource}), inventory management (see \cite{federico.ferrari.rodosthenous.2021two}),
operations research (see \cite{Guo&Kaminsky&Tomecek&Yuen11, Harrison&Taksar83}), queuing theory (see \cite{Krichagina&Taksar92}), insurance mathematics (see \cite{Lokka&Zervos08}), mathematical biology (see \cite{Alvarez&Sheep98}) and so on. 
 
All these examples represent the main motivation to study in a systematic way stochastic differential games  in which each player can control a linear stochastic differential equation in order to minimize a cost which is quadratic in the state and linear in the control. 
From the mathematical point of view, replacing the quadratic cost $c_t^i (\alpha_t^i)^2$ with a linear one (say, $c_t^i |\alpha_t^i|$), requires to introduce the so-called \emph{singular controls}: namely, to replace the additive control term $\int_0^t\alpha_s^i ds$ with some c\`adl\`ag (i.e., right continuous with left limits) bounded variation process $v_t^i$ (i.e., the singular control). 
Thus, the LQ game considered above is replaced with the game in which each player $i$ can choose a bounded variation control $v^i$ and faces the singular control problem:
\begin{align*}
\text{(LQS game)} \  
&  \Minimize\limits_{{v^i}} \ \mathbb{E} \bigg[ \int_0^T X_t Q^i_t X_t  dt + X_T Q^i_T X_T +  \int_{[0,T]} c_t^i d|v^i|_t  \bigg], \\
& \text{subject to } dX_t^{j}= ( a^j_t + b^j_t X_t^{j} )dt + \sigma_t^j dW_{t}^j + dv_t^j, \ X_{0-}^{j}=x_0^j, \ j=1,...,N.
\end{align*} 
Here $|v^j|$ denotes the total variation of the process $v^j$.
We call these games \emph{linear-quadratic-singular stochastic differential games} (LQS games, in short). 

\subsection{Background on singular control games}

A game with singular controls was first studied in  
\cite{grenadier2002option}, in order to derive symmetric equilibrium investment strategies in a continuous-time symmetric (i.e., when the cost functionals and the dynamics are the same for all players) option exercise game.
This model was later revised in \cite{Back&Paulsen}, where the open-loop equilibrium
 is provided under a suitable specification of the model.

For singular control problems, \cite{Bank&Riedel01} introduced a system of first order conditions, characterizing the optimal policies (see also \cite{Ferrari.2015.AAP, ferrari.salminen.2016irreversible}). 
These conditions represent a version of the stochastic maximum principle (see \cite{peng.90}) in the context of singular control.
Inspired by the earlier work \cite{Bank05}, \cite{Steg} consideres irreversible investment problems in oligopoly markets, determines the equilibrium in the symmetric case, and characterizes (in the non symmetric case) the open-loop equilibria through the first order conditions.
A similar approach is also followed in \cite{Ferrari&Riedel&Steg} for a public good contribution game in which  players are allowed to choose a regular control and a singular control. 
A general characterization of open-loop Nash equilibria through the stochastic maximum principle approach  has been investigated in \cite{Wang&Wang&Teo18} for regular-singular stochastic differential games. 
The existence of equilibria in a non-symmetric game  with multi-dimensional singular controls and non-Markovian costs has been established in \cite{DianettiFerrari} when the costs satisfy the submodularity conditions (see  \cite{To} for a seminal paper on static $N$-player submodular games).  
The submodularity property represents, roughly speaking, the situation in which players have an incentive to imitate the behaviour of their opponents, and it is widely used in the economic literature (see \cite{Topkis11, Vives01}).

The study of Markovian equilibria in games with singular controls seems to be particularly challenging  (see the discussion in Section 2 of \cite{Back&Paulsen}). 
Indeed, following the dynamic programming principle approach, finding a Markovian equilibrium means to construct a solution of a related reflected stochastic differential equation, on which few is known even for the control problem (see \cite{boryc&kruk2016, dianetti2021multidimensional, kruk2000}).
Hoverver, we can mention few contributions.
By showing a verification theorem, \cite{Guo&Tang&Xu18, Guo&Xu18} discuss some sufficient conditions for Nash equilibria in terms of a system of partial differential equations, and construct a Markovian equilibrium in a linear quadratic symmetric game.
When two players acts on the same one-dimensional diffusion, Nash equilibria are computed in \cite{kwon2020}, while connections between nonzero-sum games of singular control and games of optimal stopping have been tackled in \cite{DeAngelis&Ferrari18}.
We also mention \cite{Cont.Guo.Xu2020}, where Pareto optima are analysed, and \cite{bovo.deangelis.issoglio.2022variational, hernandez.simon.zervos.2015, Kwon&Zhang15} for other types of games involving singular controls. 

When the number of players is very large, Markovian equilibria can be approximated via mean field games (see \cite{HuangMalhameCaines06, LasryLions07}).
In the singular control case, the abstract existence of mean field game equilibria is studied under general conditions in \cite{Guo&Lee17, fu2023extended, Fu&Horst17} and, for submodular mean field games, in \cite{dianetti.ferrari.fischer.nendel.2022unifying}. 
A more explicit analysis is instead provided in \cite{Campietal, CaoGuo, Guo&Xu18} and in \cite{cao.dianetti.ferrari.2021}, both for the discounted infinite horizon problem and in the case of ergodic costs. 
We also mention the recent
\cite{he.tan.zou.2023mean}, which provides a representation theorem for the equilibria.

\subsection{Result and methodology} 
The main objective of this paper is to show, under fairly general assumptions, the existence of open-loop Nash equilibria for LQS games. 

The proof this result hinges on an approximation technique and on the use of the stochastic maximum principle. 
In particular, we introduce a sequence of approximating games where, for any $n\in \mathbb{N}$, players are restricted to pick strategies with Lispchitz constant bounded by $n$.  
For fixed $n$, this approximating problem can be reformulated in terms of a more standard stochastic differential game, which falls into the class of games with bang-bang controls (see \cite{hamadene.mannucci2019, hamadene.mu.2014bang, mannucci2004, mannucci2014}); that is, depending on the state of the system, players at equilibrium do nothing or act with the maximum rate allowed.  
Thanks to the results in \cite{hamadene.mannucci2019}, the existence of a Nash equilibrium $\eta^n = (\eta^{i,n},...,\eta^{N,n})$  for the game with $n$-Lipschitz controls can be established. 
By assuming some conditions on the coefficients of the matrices $Q^i$, we then show some a priori estimateson the sequence $(\eta^n)_n$. 
Indeed, this requirements on the $Q^i$ translate into a coercivity condition on the space of profile strategies and it ensures that the Nash equilibria, whenever they exist, always live in a bounded set.
These estimates allow to find an accumulation point $\eta$ of the sequence $(\eta^n)_n$.
Since, for each $n$, $\eta^n$ is a Nash equilibrium of the game with $n$-Lipschitz strategies, by the necessary conditions of the stochastic maximum principle, it can be expressed as the solution of a  certain forward-backward stochastic differential equation. 
We then take limits in such a system in order to prove that the limit point $\eta$ satisfy a set of conditions (in the spirit of the stochastic maximum principle), which in turn ensure $\eta$ to be a Nash equilibrium. 
Indeed, as a byproduct of our result, one obtains the existence a solution to the system of forward-backward stochastic differential equation related to the equilibria. 


As an application of our main result, we show the existence of equilibria in non-symmetric games of capacity expansion in oligopoly markets (see \cite{Back&Paulsen, Steg}).

\subsection{Outline of the paper}
The rest of the paper is organized as follows. 
In Section \ref{section LQS games} we introduce the probabilistic setup for LQS games and discuss some preliminary results. 
Sections \ref{section existence} is devoted to the existence theorem for Nash equilibria, while in Section \ref{subsection oligopoly investment games} we present an application to oligopoly games.

\section{Linear-quadratic-singular stochastic differential games}\label{section LQS games}

\subsection{The game} 
Fix $N \in \N$, $N \geq 2$, a finite time horizon $T\in (0,\infty)$, and consider an $N$-dimensional Brownian motion $W=(W^1,...,W^N)$,  defined on a complete probability space $(\Omega, \mathcal{F},\mathbb{P})$.
Denote by $\F= ( \cF _t)_t$ the right-continuous extension of the filtration generated by $W$, augmented by the $\P$-null sets.

Consider a game with $N$ players, indexed by $i \in \{ 1, ...,N \}$.
The filtration $\F$ represents the flow of information available to players.  
When player $i$ does not intervene,  its state $X^i$ evolves accordingly to the linear stochastic differential equation 
\begin{equation}
\label{eq SDE uncontrolled} 
    dX_t^{i}= ( a^i_t + b^i_tX_t^i )dt + \sigma^i_t dW_{t}^i,  \quad X_{0}^{i}=x^i_0.
\end{equation}   
The drift and the volatility of $X^i$ are given in terms of deterministic bounded measurable functions 
$a^i,b^i: [0,T] \to \R$ and $\sigma^i : [0,T] \to [0,\infty)$, while the initial condition $x_0^i\in \R$ is deterministic. 
Each player $i$ is allowed to choose two controls $\xi^i$ and $\zeta^i$ in the set 
\begin{equation*}
\tilde{\mathcal{A}}:= \left\{ \, \xi :\Omega \times [0,T] \rightarrow [0,\infty) \, \bigg| \,\begin{matrix} \xi \text{ is an $\mathbb{F}$-progressively measurable  c\`adl\`ag} \\ \text{  nondecreasing  process,  with $\E [ \xi_T ] < \infty $} \end{matrix} \, \right\}.
\end{equation*} 
Thus, the strategy of player $i$ is given by the vector $\eta^i = (\xi^i, \zeta^i) \in \tilde{\A}^2$.
We will denote by $\eta := (\eta^1, ..., \eta^N) \in \tilde{\A}^{2N}$ a vector of  strategies, also referred to as \emph{profile strategy}.
Given strategies $ \xi^j, \zeta^j \in \tilde \A $, $j=1,...,N$, with slight abuse we will interchangeably use  the notations $((\xi^1, \zeta^1),...,(\xi^N,\zeta^N)) = (\xi^1,...,\xi^N,\zeta^1,...,\zeta^N) = (\xi, \zeta)$.
Also, for a profile strategy $\eta \in \tilde \A ^{2N}$ and deviations $\bar \xi^i, \bar \zeta ^i \in \tilde \A$, we define the unilateral deviation for player $i$ as $(\eta,\eta^{-i})=(\xi,\zeta, \eta^{-i}) = ((\eta,\eta^{-i})^1,...,(\eta,\eta^{-i})^N)$ with
$$
(\eta,\eta^{-i})^j := 
\begin{cases}
   \eta^j & \text{if} \quad j \ne i, \\
   (\bar \xi ^i, \bar \zeta ^i) & \text{if} \quad  j = i.
\end{cases}
$$
A strategy $\eta^i = (\xi^i, \zeta^i) \in \tilde{\A}^2$ is said to be \emph{admissible} if it is an element of
$$
\A ^2:= \Big\{ \text{$\eta^i = (\xi^i, \zeta^i) \in \tilde{\A}^2$ with $v^i := \xi^i - \zeta^i$ satisfying }  \begin{matrix} \mathbb E \big[ \int_0^T |v_t^i|^2 dt  + |v_T^i|^2 \big] < \infty \end{matrix} \Big\}.
$$
Similarly, an admissible profile strategy is a vector $\eta \in \A ^{2N}$.

When the admissible profile strategy $\eta \in  \A^{2N}$ is chosen by the players, the controlled state $X^\eta:=(X^{1,\eta},...,X^{N,\eta})$ of the system evolves as 
\begin{equation}
\label{eq SDE}
    dX_t^{i,\eta}= ( a^i_t + b^i_t X_t^{i,\eta})dt + \sigma_t^i dW_{t}^i + d\xi_{t}^i - d \zeta^i_t, \quad X_{0-}^{i,\eta}=x_0^i, \quad i=1,...,N, 
\end{equation} 
where $X_{0-}^{i,\eta}$ denotes the left limit in $0$ of the process $X^{i,\eta}$.
Notice that the effect of the controls of the players is linear on the state, and that, for any $\eta \in \A^{2N}$, there exists a unique strong solution $X^\eta$ (we refer to \cite{Protter05} for further details).

Given admissible strategies $\eta^{-i} \in \A^{2(N-1)}$, the aim of player $i$ is to choose $\eta^i = (\xi^i,\zeta^i) \in \A^2$ in order to minimize the quadratic-singular expected cost
\begin{equation}\label{eq cost functional}
J^i(\eta^i,\eta^{-i}) := \mathbb{E} \bigg[ \int_0^T X_t^\eta Q^i_t X_t^\eta dt + X_T^\eta Q^i_T X_T^\eta + \int_{[0,T]} (c_t^{i,+} d\xi^{i}_{t} + c_t^{i,-} d\zeta^{i}_{t})   \bigg].
\end{equation}  
Here, the $\R^{N \times N}$ matrix $Q_t^i= (q_t^{k,j;i})_{k,j}$ is given via bounded measurable functions $q^{k,j;i}:[0,T] \to \R$, $k,j=1,...,N$, and we set
$$
y Q^i_t z := \sum_{j,k=1}^N q_t^{k,j;i} y^k z^j, \quad y, z \in \R^N.
$$ 
The cost of increasing and decreasing the state process is given by continuous functions $c^{i,+}, c^{i,-} :[0,T] \to [0,\infty)$, respectively.
Also, since any c\`adl\`ag bounded variation process $v$ can be identified with a Radon measure on $[0,T]$, for any continuous function $f:[0,T] \to \R$, the integrals with respect to $v$ are defined by
$$ 
\int_{[0,T]} f_t \, dv_t:= f_0 v_0 + \int_0^T f_t \, dv_t, 
$$
where the integral on the right hand side is intended in the standard Lebesgue-Stieltjes sense on the interval $(0,T]$.
Notice that, in light of the square integrability of the admissible strategies (see the definition of $\A^2$), the cost functional is well defined.

We will focus on the following notion of equilibrium.  
\begin{definition}
An admissible profile strategy $ \eta \in \A ^{2N}$ is an (open-loop) Nash equilibrium if
$$ 
J^i (  \eta ^i,  \eta ^{-i}) \leq J^i ( \bar \eta^i,  \eta ^{-i}), \quad \text{for any $\bar \eta^i \in \A ^2,$}
$$
for any $i =1,...,N$.
\end{definition}

\subsection{Stochastic maximum principle for LQS games} 
For later use, we now review some basic tools in the theory of stochastic singular control.
In particular, we will introduce the adjoint processes and state a version of the stochastic maximum principle.

For a generic $d \in \mathbb N$, define the set
$$
\mathbb H ^{2,d}:= \Big\{ \text{$M: \Omega \times [0,T] \to \R^d \, \Big| \, \mathbb F$-progr.\ meas.\ process with }  \begin{matrix} \mathbb E \big[ \int_0^T |M_t|^2 dt \big] < \infty \end{matrix} \Big\}
$$
and set $\mathbb H ^{2}:=\mathbb H ^{2,1}$.

Given an admissible profile strategy $\eta$, for any $i=1,...,N$, define the adjoint process $Y^{i,\eta} \in \mathbb H ^{2}$ as
\begin{equation}\label{eq BSDE adjoint}
    Y_t^{i,\eta} := 2 \E \bigg[ \Gamma^i_{t,T} Q^{i;i}_T X^\eta_T + \int_t^T \Gamma^i_{t,s}  Q^{i;i}_s X^\eta_s ds    \bigg| \mathcal F _t \bigg],  
     \quad t \in[0,T], 
\end{equation}
where
\begin{equation}\label{eq Gamma}
\Gamma_{s,t}^i : = \exp \Big( \int_s^t b^i_r dr \Big), \quad s,t \in [0,T].
\end{equation}
Here and in the sequel, for $i,j=1,...,N$, the (raw) vector $Q_t^{j;i} = (q_t^{j,1;i},...,q_t^{j,N;i})$ denotes the $j$-th raw of the matrix $Q^i_t$, and $Q_t^{j;i} x$ is the product $\sum_{k=1}^N q_t^{j,k;i} x^k$, $x\in \R^N$.
\begin{remark}
For any $i=1,...,N$, the process $Y^{i,\eta}$ admits a continuous version, and it coincides with the component $Y^i$ of the solution $(Y,Z) = (Y^1,...,Y^N, Z^1,...,Z^N)\in \mathbb H^{2,N} \times \mathbb H ^{2,N \times N}$ to the linear backward stochastic differential equation (BSDE, in short) 
$$ 
d Y ^j_t = - (  2 Q^{j;i}_t X^\eta_t + {b^j_t} Y_t^j) dt + Z_t^j dW_t,  \quad Y_T^j = 2 Q^{j;i}_T X^\eta_T, \quad j=1,...,N,
$$ 
which admits explicit solution as in \eqref{eq BSDE adjoint} (see Proposition 6.2.1 at p.\ 142 in \cite{pham2009continuous}).
Indeed, the driver of such a BSDE consists of the partial derivatives of the pre-Hamiltonians, and the solution of such a BSDE is typically referred to as adjoint process (see \cite{Wang&Wang&Teo18} for further details). 
\end{remark}

The ajoint process $Y^{i,\eta}$  can also be interpreted as the subgradient of the cost functional $J^i (\cdot, \eta^{-i})$. 
This observation is made rigorous in the following lemma.
\begin{lemma}\label{lemma subgradient identities}
For any $i =1,...,N$,  $\eta = (\eta^1, ...,\eta^N) = (\xi,\zeta) \in \A^{2N}$ and $(\b \xi ^i, \b \zeta ^i) \in \A^2$, we have
\begin{align*}
 J^i(\b \xi ^i, \b \zeta ^i; \eta^{-i}) - J^i(\xi ^i, \zeta ^i; \eta^{-i}) \geq & \E \bigg[ \int_{[0,T]} Y_t^{i,\eta}d(\b v^i - v^i)_t \\
 &  +  \int_{[0,T]} c_t^{i,+} d(\b \xi^i - \xi^i)_t + \int_{[0,T]} c_t^{i,-}d ( \b \zeta^i - \zeta^i)_t  \bigg],   
\end{align*} 
where $v^i : = \xi^i - \zeta^i$ and $ \b v ^i : = \b \xi ^i - \b \zeta ^i$.
\end{lemma} 
\begin{proof}
Take $i \in \{1,...,N\}$, $\eta = (\eta^1, ..., \eta^N) \in \A^{2N}$ and $\b \xi ^i, \b \zeta ^i \in \A$.
In order to simplify the notation, set $ X := X^\eta, \ Y^i := Y^{i,\eta}$ and denote by $\b X^i$ the solution to the SDE
$$
d \b X_t^{i}= ( a^i_t + b^i_t \b X_t^{i})dt + \sigma_t^i dW_{t}^i + d \b \xi _{t}^i - d \b \zeta ^i_t, \quad \b X _{0-}^{i}=x^i_0.
$$

For later use, we first show the following elementary identity:
\begin{align}\label{eq key idenity}
\E \bigg[ \int_{[0,T]} Y^i_t  d (\b v^i- v^i)_t \bigg] 
= \E \bigg[ \int_0^T  2 Q_t^{i;i} X_t (\b X^i_t - X ^i_t) dt + 2 Q_T^{i;i} X_T (\b X^i_T - X ^i_T) \bigg]. 
\end{align} 
Indeed, recalling the definition of $(\Gamma^i_{t,s})_{t,s}$ in \eqref{eq Gamma} and using that
$$
\Gamma^i_{t,0} ( \b X^i_t - X ^i_t) = \int_{[0,t]} \Gamma^i_{s,0} d( \b v^i- v^i)_s,
$$
via an integration by parts we obtain
\begin{align}\label{eq intermediate key}
 \E \bigg[ \int_0^T 2Q_t^{i;i} & X_t^i ( \b X^i_t - X ^i_t) dt + 2 Q_T^{i;i} X_T (  \b X^i_T - X ^i_T) \bigg]\\ \notag
 &= \E \bigg[ \int_0^T \Big( -\int_t^T \Gamma_{0,s}^i 2 Q_s^{i;i} X_s ds \Big)' \Big( \int_{[0,t]} \Gamma^i_{s,0} d( \b v^i- v^i)_s \Big) dt \\ \notag
& \quad \quad +  \Gamma_{0,T}^i 2 Q_T^{i;i} X_T \Big( \int_{[0,T]} \Gamma^i_{t,0} d( \b v^i- v^i)_t \Big) \bigg] \\ \notag
& = \E \bigg[ \int_{[0,T]} \Big( \int_t^T \Gamma_{t,s}^i 2 Q_s^{i;i} X_s ds +  \Gamma_{t,T}^i2 Q_T^{i;i} X_T \Big)  d ( \b v^i- v^i)_t \bigg]. \notag
\end{align}
Moreover, Theorem 1.33 in \cite{jacod1979calcul},  implies that
\begin{align*}
&\E \bigg[ \int_{[0,T]} \Big( \int_t^T \Gamma_{t,s}^i 2 Q_s^{i;i} X_s ds +  \Gamma_{t,T}^i 2 Q_T^{i;i} X_T  \Big)  d ( \b v^i- v^i)_t \bigg] \\
& \quad = \E \bigg[ \int_{[0,T]} \E \Big[ \int_t^T \Gamma_{t,s}^i 2 Q_s^{i;i} X_s ds +  \Gamma_{t,T}^i2 Q_T^{i;i} X_T \Big| \cF _t \Big]  d ( \b v^i- v^i)_t \bigg], 
\end{align*}
which, together with \eqref{eq intermediate key} and \eqref{eq BSDE adjoint}, gives \eqref{eq key idenity}.

Now, by the convexity of the maps $x \mapsto x Q^i_t x$,  thanks to \eqref{eq key idenity} we have
\begin{align*}
J^i(\b \xi ^i, \b \zeta ^i; \eta^{-i}) &- J^i(\xi ^i, \zeta ^i; \eta^{-i}) \\
& \geq \E \bigg[ \int_0^T  2 Q_t^{i;i} X_t ( \b X^i_t -  X ^i_t) dt + 2 Q_T^{i;i} X_T ( \b X^i_T -  X ^i_T)  \bigg] \\
& \quad \quad  + \E \bigg[ \int_{[0,T]}  c_t^{i,+} d( \b \xi^i -  \xi^i)_t + \int_{[0,T]}  c_t^{i,-} d( \b \zeta^i -  \zeta^i)_t  \bigg]\\
& =  \E \bigg[ \int_{[0,T]} Y_t^{i,\eta}d(\b v^i - v^i)_t \bigg] \\
 &  \quad \quad  + \E \bigg[ \int_{[0,T]}  c_t^{i,+} d( \b \xi^i -  \xi^i)_t + \int_{[0,T]}  c_t^{i,-} d( \b \zeta^i - \zeta^i)_t  \bigg],
\end{align*}
completing the proof of the lemma.

\end{proof}

Next, we state the following version of the stochastic maximum principle, characterizing the Nash equilibria in terms of the related adjoint processes. 
Such a theorem was originated in \cite{Bank&Riedel01} for singular control problems, and we refer to \cite{Wang&Wang&Teo18} for a more general version in a game-context (which contains Theorem \ref{thm SMP} below as a particular case).
In light of Lemma \ref{lemma subgradient identities}, the proof of the following result is straightforward, and it is therefore omitted. 
\begin{theorem}\label{thm SMP}
    The admissible profile strategy $ \eta = (\xi, \zeta) \in\A^{2N}$ is a Nash equilibrium if and only if, 
    for any $i=1,...,N$, the following conditions hold:
    \begin{enumerate}
        \item $ Y ^{i,\eta}_t + c_t^{i,+} \geq 0$ and $-  Y ^{i,\eta}_t + c_t^{i,-} \geq 0$, for any $t \in [0,T], \ \mathbb P$-a.s.;
        \item $\int_{[0,T]} ( Y ^{i,\eta}_t + c_t^{i,+}) d \xi ^i_t = 0$ and $\int_{[0,T]} (- Y ^{i,\eta}_t + c_t^{i,-}) d \zeta ^i_t =0$, $\mathbb P $-a.s.
    \end{enumerate} 
\end{theorem}

\section{Existence of Nash equilibria}\label{section existence}

\subsection{Assumptions and main result}
Define  matrix-valued functions $ \hat Q, \b Q :[0,T] \to \R^{N\times N}$ as 
\begin{equation}
    \label{eq def matrix bar Q}
\hat Q _t^{k,j} =  2 q_t^{k,j;k}, 
\quad \text{and} \quad  
\bar Q _t^{k,j} =
\begin{cases}
    q_t^{k,k;k} & \text{if} \quad k = j, \\ 
     2 q_t^{k,j;k}   & \text{if} \quad { k\ne j}.
\end{cases}
\end{equation}

We now summarize the sufficient conditions for the existence of Nash equilibria.
\begin{assumption}\label{assumption} For any $i=1,...,N$, we require that:
\begin{enumerate}
    \item The functions $a^i, b^i, \sigma^i, q^{k,j;i}: [0,T] \to \mathbb{R}$ are bounded, for any $k,j=1,...,N$;
    \item\label{condition c} The functions $c^{i,+}, c^{i,-}: [0,T] \to  (0,\infty)$ are continuous;
    \item\label{condition Q i symmetric} For any $t\in [0,T]$, the matrix $Q_t^i$ is symmetric; 
    \item\label{condition Q bar Q >0} For any $t\in [0,T]$, the matrix $\bar Q _t$ is positive definite (hence, also $\hat Q _t$ is positive definite); i.e., there exists $\kappa > 0$ such that $ x\bar Q _t x \geq \kappa |x|^2$ for any $x\in \R^N$.
\end{enumerate}
\end{assumption} 

It is worth to underline that some of these requirements are in place for convenience of exposition (in particular, the symmetry of $Q^i$): a model in which these are violated is discussed in Section \ref{subsection oligopoly investment games}. 
\begin{remark}
Clearly, the most restrictive hypothesis is the positive definiteness of the matrices $\b Q ^i$.
On the one hand, it implies that, for any $x \in \R^N$, we have $x Q ^i x \geq C (|x^i|^2 - |x^{-i}|^2)$ for some $C>0$. 
This condition is quite standard in singular control (see \cite{dianetti2021multidimensional}, among others) and it allows to prove the existence of the optimal controls for the single-player optimization problems (i.e., for the control problems $\inf_{\eta^i} J^i(\eta^i, \eta^{-i})$, parametrized by $\eta^{-i}$).
On the other hand, the assumption on the matrix $\b Q$ represents a coercivity condition on the space of profile strategies and it ensures that the Nash equilibria, whenever they exist, always live in a bounded subset of $\A^{2N}$ (see the a priori estimates in Lemma \ref{lemma a priori estimates} below).  
This assumption is different from more typical requirements in LQ games, which instead implies a certain monotonicity of the associated forward-backward system of equations (see Sections 5.2.2 and 5.4.3 in \cite{Carmona2016}). 
\end{remark}

We now state the main result of this paper.
\begin{theorem}\label{thm main}
Under Assumption \ref{assumption}, there exists a Nash equilibrium. 
\end{theorem}

The proof of Theorem \ref{thm main} is given in the next subsection (see Subsection \ref{subsection proof of thm main}), and it consists of several steps. 
We resume here the key ideas. 
First, we introduce a sequence of approximating games where, for any $n\in \mathbb{N}$, players are restricted to pick strategies $\xi^i, \zeta^i \in \tilde{\A}$ with Lispchitz constant bounded by $n$. 
For fixed $n$, this approximating problem falls into the class of games with bang-bang controls and we can employ the results in \cite{hamadene.mannucci2019} in order to show the existence of a Nash equilibrium $\eta^n = (\eta^{i,n},...,\eta^{N,n})$.
We then show some a priori estimates on the sequence $(\eta^n)_n$, which in turn allow to find an accumulation point $\eta$.
Finally, we prove that the limit point satisfies the conditions of Theorem \ref{thm SMP}, hence it is a Nash equilibrium.

\subsection{Proof of Theorem \ref{thm main}}\label{subsection proof of thm main} 
In the following subsections we will prove Theorem \ref{thm main}, and Assumption \ref{assumption} will be in force. 
During the proofs, $C>0$ will denote a generic constant, which might change from line to line. 

\subsubsection{Nash equilibria for a sequence of approximating games} 
Define, for each $n\geq1$, the $n$-\emph{Lipschitz game} as the game in which, for any $i=1,...,N$,  player $i$ is allowed to chose strategies $\xi^i$ and $\zeta^i$ in the space of $n$-Lipschitz strategies 
$$
\A _n :=\{ \xi \in \tilde{\A} \text{ with Lipschitz constant bounded by $n$ and $\xi_0=0$} \}.
$$
For a given profile strategy $\eta= (\xi, \zeta) \in \A_n^{2N}$, player $i$ minimizes the cost $J^i$ defined as in \eqref{eq cost functional}, in which the state equation is replaced by the controlled SDE
\begin{equation}\label{eq SDE n}
 dX_t^{i,\eta}=( a^i_t + b^i_t X_t^{i,\eta} )ds + \sigma^{i,n}_t dW_{t}^i + d\xi_{t}^i - d \zeta^i_t, \quad X_{0}^{i,\eta}=x^i_0, 
\end{equation}   
with strictly elliptic diffusion term 
\begin{equation}
    \label{eq ellipticity n} 
     \sigma_t^{i,n} := \sigma_t^i \lor \frac1n.
\end{equation}

We first have the following existence result.
\begin{proposition}\label{prop NE Lip}
For any $n \geq 1$, there exists a Nash equilibrium $\eta ^n = ( \eta ^{1,n},..., \eta ^{N,n}) \in \A_n^{2N}$ of the $n$-Lipschitz game; that is, $\eta ^n  \in \A_n^{2N}$ such that
$$ 
J^i (  \eta ^{i,n},  \eta ^{-i, n}) \leq J^i (\bar \eta ^{i}, \eta ^{-i,n}), \quad \text{for any $\bar \eta ^i \in \A_n^2,$}
$$
for any $i =1,...,N$.
\end{proposition}

\begin{proof} 
For each $n$, the $n$-Lipschitz game can be reformulated as a stochastic differential game with regular controls by setting
$$
u_t^i:=\frac{d\xi_t^i}{d t} 
\quad \text{and} \quad
w_t^i:=\frac{d\zeta_t^i}{d t}.
$$
Indeed, in the $n$-Lipschitz game player $i$ chooses a strategies $u^i, w^i$ in the set
$$
\mathcal{U}_n:=\{ \text{progressively measurable processes $u$ with $0 \leq u_t \leq n, \ \mathbb{P}\otimes dt$-a.e.} \}
$$ 
in order to minimize the expected cost
\begin{align}\label{eq SDG regular controls}
& J^i(\alpha^i,\alpha^{-i}) := \mathbb{E} \bigg[ \int_0^T \big( X_t^\alpha Q^i_t X_t^\alpha + c_t^{i,+} u^i_t + c_t^{i,-} w^i_t \big) dt
+ X_T^\alpha Q^i_T X_T^\alpha \bigg],\\ \notag
&\text{subject to} \quad dX_t^{j,\alpha}=( a^j_t + b^j_t X_t^{j,\alpha} + u_t^j - w_t^j)ds + \sigma^{j,n}_t dW_{t}^j, \ X_{0}^{j,\alpha}=x^j_0, \ j=1,...,N. \notag  
\end{align} 
Here, we use the notation  
$$
\alpha : =(\alpha^1,...,\alpha^N) := ((u^1,w^1),...,(u^N,w^N)) \quad \text{and} \quad X^{\alpha} := (X^{1,\alpha},...,X^{N,\alpha}).
$$ 
Thanks to the uniform ellipticity enforced in \eqref{eq ellipticity n}, we can employ Theorem 4.1 in \cite{hamadene.mannucci2019} in order to deduce that, for any $n$ there exists a Nash equilibrium 
 $\alpha^n = (\alpha^{1,n},..., \alpha^{N,n})$, with 
$\alpha^{i,n} = (u^{i,n}, w^{i,n}) \in \mathcal U _n^2$.
Hence, defining the processes
\begin{equation}\label{eq xi n zeta n first time}
\xi^{i,n}_t := \int_0^t u_s^{i,n} ds \quad \text{and} \quad \zeta^{i,n}_t := \int_0^t w_s^{i,n} ds,  
\end{equation}
we have that $\eta^n :=(\eta^{1,n},..., \eta^{N,n})$ is a Nash equilibrium for the $n$-Lipschitz game, with $\eta^{i,n} := (\xi^{i,n}, \zeta^{i,n})$.
\end{proof}

For any $n \in \mathbb N$, we can now fix a Nash equilibrium $\eta^n \in \mathcal A _n^{2N}$, which is given in terms of the equilibrium $\alpha^n$ of the game in \eqref{eq SDG regular controls}.
Thus, we proceed by characterizing such these equilibria by using the stochastic maximum principle (see Chapter 5 in \cite{Carmona2016}).

When viewed as a stochastic differential game with regular controls (see \eqref{eq SDG regular controls}), the pre-Hamiltonians of the $n$-Lipschitz game write as
$$
H^{i,n}(t,x,u, w ;y^1,...,y^N):=  \sum_{j=1}^N (a^j_t  + b^j_t x^j + u^j -w^j) y^j + x Q^i_t x + c_t^{i,+} u^i + c_t^{i,-} w^i, 
$$
for any $i=1,...,N$, $(t,x) \in[0,T] \times \R^N$, $u=(u^1,...,u^N), \, w=(w^1,...,w^N)  \in [0,n]^N$ and $y:=(y^1,...,y^N) \in \R^N$.
In particular, the function $H^{i,n}$ represents the pre-Hamiltonian related to the optimization problem of player $i$.

Define the process $( X ^{n}, Y ^{n}, Z ^{n}) = (X^{1,n},...,X^{N,n}, Y^{1,n},...,Y^{N,n},Z^{1,n},...,Z^{N,n})$, in which
$ (X^{i,n},Y^{i,n},Z^{i,n}) \in \mathbb H^2 \times \mathbb H^{2,N} \times \mathbb H^{2,N\times N}$, as the unique solution of the forward-backward stochastic differential equation (FBSDE, in short)
$$
\begin{cases}\label{eq first FBSDE n}   
d X^{i,n}_t & = \big(  a^i_t + b^i_t X ^{i,n}_t + u_t^{i,n} - w_t^{i,n} \big)dt  + \sigma ^{i,n}_tdW^i_t, \quad X_{0}^{i,n}=x^i_0, \quad i=1,...,N,\\ 
d Y^{i,n}_t  & = -  D_x H^{i,n} ( t, X_t^{n}, Y _t^{n}) dt + Z_t^{i,n} dW_t, \quad Y_T^{i,n} = 2 Q^{i}_T X_T^{n}, \quad i=1,...,N.  
\end{cases}
$$
The necessary conditions of the stochastic maximum principle (see Theorem 5.19 at p.\ 187 in \cite{Carmona2016}) characterize the Nash equilibria as the minimizers of the pre-Hamiltonian; that is, for any $i=1,...,N$ we have  
\begin{align}\label{eq first necess cond regular} 
 H^{i,n} & (t,X_t^n, u^{1,n}_t,...,u_t^{N,n}, w_t^{1,n},...,w_t^{N,n} ;Y_t^{1,i,n},...,Y_t^{N,i,n}) \\ \notag
& = \inf _{ u^i, w^i \in [0,n] } H^{i,n}(t,X_t^n, (u^i,u_t^{-i,n}), (w^i,w_t^{-i,n}) ;Y_t^{1,i,n},...,Y_t^{N,i,n}), \quad \mathbb P \otimes dt\text{-a.e.} \notag
\end{align}
Now, we can compute the optimal feedbacks for player $i$, as multivalued functions $ \hat u ^{i,n}$ and $\hat w ^{i,n}$
from $[0,T] \times \R^{K+N} \times [0,n]^{2(N-1)} \times \R^N $ into $[0,n]$.
Indeed, by setting
\begin{align*}
(\hat u ^{i,n}, \hat w^{i,n}) (t,x,u^{-i}, w^{-i} ;&y^1,...,y^N)\\
&:=\argmin_{ u^i, w^i \in [0,n] } H^{i,n}(s,x,u^1,...u^N, w^1,...,w^N ;y^1,...,y^N)  \\
&:= \Big( \argmin_{ u^i \in [0,n] } u^i(y^i + c_t^{i,+}), \argmin_{ w^i \in [0,n] } w^i(-y^i + c_t^{i,-}) \Big), 
\end{align*}
and using the notation $n A := \{ n a \, | \, a \in A \}$ for $A \subset \R$, we obtain
\begin{align*}
\hat u ^{i,n}(t,y^{i}) = n
\begin{cases}
    \{ 1 \} & \text{ if } y^i + c_t^{i,+} < 0 , \\
    [0,1]   & \text{ if } y^i + c_t^{i,+} = 0 , \\
    \{0 \}  & \text{ if } y^i + c_t^{i,+} > 0 , \\
\end{cases}
\end{align*}
and
\begin{align*}
\hat w ^{i,n}(t,y^i)= n
\begin{cases}
    \{ 1 \} & \text{ if } - y^i + c_t^{i,-} < 0 , \\
    [0,1]   & \text{ if } - y^i + c_t^{i,-} = 0 , \\
    \{0 \}  & \text{ if } - y^i + c_t^{i,-} > 0 . \\
\end{cases}
\end{align*}
Hence, the necessary conditions \eqref{eq first necess cond regular} rewrites as
\begin{equation}\label{eq necess cond regular} 
u _t^{i,n} \in \hat u^{i,n} (t,Y_t^{i,i,n})
\quad \text{and} \quad
w _t^{i,n}\in \hat w^{i,n} (t,Y_t^{i,i,n}), \quad \mathbb P \otimes dt\text{-a.e.}.
\end{equation}

Since $\hat u ^{i,n}$ and $\hat w ^{i,n}$ depend only on $Y^{i,i,n}$ and since the equation for $Y^{i,i,n}$ does not depend on $Y^{j,i,n}$ if $j \ne i$, one can reduce the FBSDE \eqref{eq first FBSDE n}.
In particular, with slight abuse of notation (i.e., writing $(Y ^{i,n}, Z ^{i,n})$ instead of $(Y ^{i,i,n}, Z ^{i,i,n})$) we have $ u _t^{i,n} = \hat u ^{i,n}(t, Y_t^{i,n})$ and $w _t^{i,n} = \hat w^{i,n}(t, Y_t^{i,n})$ with 
$( X ^{n},  Y ^{n},  Z ^{n})= (X^{1,n},...,X^{N,n}, Y ^{1,n}, ...,Y ^{N,n}, Z^{1,n}, ..., Z ^{N,n})$ solution to the FBSDE
\begin{equation*}
\begin{cases}   
d  X^{i,n}_t & = \big( a^i_t + b^i_t X ^{i,n}_t+ u_t^{i,n} -w_t^{i,n} \big) dt  + \sigma ^{i,n}_t dW^i_t,  \quad X_{0}^{i,n}=x^i_0, \quad i=1,...,N,\\
d  Y^{i,n}_t  & = - ( 2 Q_t^{i;i}  X _t^{i,n} + b_t^i Y_t^{i,n}) dt +  Z _t^{i,n} dW_t,  \quad Y_T^{i,n} = 2 Q^{i;i}_T X_T^{n}, \quad i=1,...,N. \\  
\end{cases}
\end{equation*} 
Moreover, using \eqref{eq xi n zeta n first time} and noticing that the equations for $Y^{i,n}$ are linear, by using Proposition 6.2.1 at p.\ 142 in \cite{pham2009continuous}, we can rewrite this system as
\begin{equation}\label{eq FBSDE n}
\begin{cases}   
d  X^{i,n}_t & = \big( a^i_t + b^i_t X ^{i,n}_t \big) dt  + \sigma ^{i,n}_t dW^i_t  + d  \xi^{i,n}_t - d \zeta^{i,n}_t, \quad X_{0}^{i,n}=x_0^i, \quad i=1,...,N,\\
Y_t^{i,n} &= 2 \E \Big[ \Gamma^i_{t,T} Q^{i;i}_T X^n_T + \int_t^T \Gamma^i_{t,s}  Q^{i;i}_s X^n_s ds    \Big| \mathcal F _t \Big], \quad i=1,...,N, 
\end{cases}
\end{equation} 
and the necessary conditions for the $n$-Lipschitz game (see \eqref{eq necess cond regular}) translate into
\begin{equation}
    \label{eq necessary cond n}
     \xi^{i,n}_t = \int_0^t u_s^{i,n} ds, \ u_t^{i,n} \in \hat u^{i,n} (t,Y_t^{i,n})
    \quad \text{and} \quad 
    \zeta^{i,n}_t = \int_0^t w_s^{i,n} ds, \ w_t^{i,n} \in \hat w^{i,n} (t,Y_t^{i,n}). 
\end{equation}

\subsubsection{A priori estimates for Nash equilibria and convergence to a limit point}  

From the previous subsection, we can fix an extended sequence of Nash equilibria $(X^n,Y^n,\xi^n,\zeta^n)_n$ of the $n$-Lipschitz games, which solves the FBSDE \eqref{eq FBSDE n} and satisfy the conditions \eqref{eq necessary cond n}.

We begin with the following a priori estimates on the moments of the these Nash equilibria. 
\begin{lemma}\label{lemma a priori estimates} 
We have 
$$
\sup_n \mathbb E \bigg[ \int_0^T |X_t^{n}|^2 dt + | X_T^{n}|^2 \bigg] < \infty
\quad \text{and} \quad\sup_n  \mathbb E \bigg[ |\xi^n_T|  + |\zeta^n_T| + \sup_{t\in[0,T]} |X_t^n| \bigg] < \infty.
$$
\end{lemma}
 
\begin{proof} 
We divide the proof in two steps. 
\smallbreak \noindent
\emph{Step 1.}
For $i=1,...,N$, let $\tilde X ^{i,n}$ denote the solution to the SDE \eqref{eq SDE n} controlled by $0$; that is, to the SDE
\begin{equation}\label{eq SDE n uncontrolled}
 d \tilde X_t^{i,\eta}=( a^i_t + b^i_t \tilde X_t^{i,\eta} )ds + \sigma^{i,n}_t dW_{t}^i , \quad \tilde X _{0}^{i,\eta}=x^i_0.
\end{equation}  
Since $\eta^n$ is a Nash equilibrium for the $n$-Lipschitz game, we have $J^i(\eta^{i,n}, \eta^{-i,n}) \leq J^i (0,0, \eta^{-i,n})$, from which we obtain
\begin{align*}
    \mathbb E \bigg[ \int_0^T & X ^{n}_t Q_t^i X ^{n}_t dt + X ^{n}_T Q^i_T X^n_T \bigg] \\ \notag
    & \leq J^i(\eta^{i,n}, \eta^{-i,n})  \leq J^i(0,0, \eta^{-i,n}) \\ \notag
    & = \mathbb E \bigg[ \int_0^T \Big( q^{i,i;i}_t (\tilde X_t^{i,n})^2 + 2 \sum_{j \ne i} q^{i,j;i}_t \tilde X_t^{i,n} X_t^{j,n} + \sum_{k\ne i,j \ne i } q^{k,j;i}_t X_t^{k,n} X_t^{j,n} \Big) dt \bigg] \\ \notag
    &\quad  +  \mathbb E \bigg[   q^{i,i;i}_T (\tilde X_T^{i,n})^2 + 2 \sum_{j \ne i}  q^{i,j;i}_T \tilde X_T^{i,n} X_T^{j,n} + \sum_{k\ne i,j \ne i } q^{k,j;i}_T X_T^{k,n} X_T^{j,n} \bigg], \notag 
\end{align*}
which in turn rewrites as 
\begin{align*}
\mathbb E \bigg[ \int_0^T & \Big( q^{i,i;i}_t ( X_t^{i,n})^2 + 2 \sum_{j \ne i} q^{i,j;i}_t  X_t^{i,n} X_t^{j,n} \Big) dt    + q^{i,i;i}_T ( X_T^{i,n})^2 + 2 \sum_{j \ne i} q^{i,j;i}_T  X_T^{i,n} X_T^{j,n} \bigg] \\
& \leq
 \mathbb E \bigg[ \int_0^T \Big( q^{i,i;i}_t (\tilde X_t^{i,n})^2 + 2 \sum_{j \ne i} q^{i,j;i}_t \tilde X_t^{i,n} X_t^{j,n}\Big) dt \\
&\quad \quad \quad \quad \quad \quad \quad \quad +  q^{i,i;i}_T (\tilde X_T^{i,n})^2 + 2 \sum_{j \ne i} q^{i,j;i}_T \tilde X_T^{i,n} X_T^{j,n} \bigg].
\end{align*}
Therefore, summing over $i=1,...,N$, for $\bar Q$ as in \eqref{eq def matrix bar Q}, 
$$
\tilde Q _t^{k,j} := 
\begin{cases}
    0 & \text{if} \quad k = j, \\ 
     2 q_t^{k,j;k}   & \text{if} \quad { k\ne j}.
\end{cases} 
\quad \text{and} \quad
\tilde X ^n:=(\tilde X^{1,n}, ...,\tilde X ^{N,n}),
$$
using the integrability of $\tilde X ^n$, we find
\begin{align*}
\mathbb E \bigg[ \int_0^T  X_t^{n} \b Q _t X_t^{n}  dt   +   X_T^{n} \b Q _T X_T^{n} \bigg]  
\leq C \bigg( 1+
 \mathbb E \bigg[ \int_0^T  \tilde X _t^n \tilde Q _t X _t^{n} dt +  \tilde X_T ^n \tilde Q _T X_T^{n} \bigg] \bigg),    
\end{align*}
and, since $\b Q > 0$ (cf.\ Condition \ref{condition Q bar Q >0} in Assumption \ref{assumption}),  we deduce that
\begin{align*} 
\mathbb E \bigg[ \int_0^T |X_t^{n}|^2 dt + | X_T^{n}|^2 \bigg]  
   \leq C \bigg( 1+
 \mathbb E \bigg[ \int_0^T  \tilde X _t^n \tilde Q _t X_t^{n} dt +  \tilde X_T ^n \tilde Q _T X_T^{n} \bigg] \bigg).    
\end{align*}
By employing H\"older inequality with exponent $2$ on the latter estimate, we obtain
\begin{align*}
\mathbb E \bigg[ \int_0^T |X_t^{n}|^2 dt + | X_T^{n}|^2 \bigg]  
 & \leq C \bigg( 1+  \sum_{i,j=1}^N  \mathbb E \bigg[\int_0^T |\tilde X^{i,n}_t|| X ^{j,n}_t | dt +  |\tilde X^{i,n}_T|| X ^{j,n}_T |\bigg] \bigg) \\ 
  & \leq C \bigg( 1 + \sum_{i,j=1}^N \bigg( \mathbb E \bigg[\int_0^T |\tilde X^{i,n}_t|^{2}   dt + |\tilde X^{i,n}_T|^{2} \bigg] \bigg)^{\frac{1}{2}} \\
  & \quad \quad \quad \quad \quad \quad \times \bigg( \mathbb E \bigg[ \int_0^T | X ^{j,n}_t |^2 dt + | X ^{j,n}_T |^{2} \bigg] \bigg)^{\frac12} \bigg).
\end{align*}
Hence, again by the integrability of $\tilde X ^n$, we get 
\begin{align*}
\mathbb E \bigg[ \int_0^T |X_t^{n}|^2 dt + | X_T^{n}|^2 \bigg]  
 \leq C \bigg( 1+  \bigg( \mathbb E \bigg[ \int_0^T | X ^{n}_t |^{2} dt + | X ^{n}_T |^{2} \bigg] \bigg)^{\frac12} \bigg),
\end{align*}
which in turn implies that
\begin{equation}
    \label{eq a priori estimates int X}
    \sup_n \mathbb E \bigg[ \int_0^T |X_t^{n}|^2 dt + | X_T^{n}|^2 \bigg] < \infty,
\end{equation}
thus proving the first part of the statement. 

\smallbreak \noindent 
\emph{Step 2.}
We now estimate $\xi_T^n$ and $\zeta_T^n$.
Fix $i\in\{1,...,N\}$ and, for $\tilde X ^{i,n}$ denoting the solution to the SDE \eqref{eq SDE n uncontrolled}, 
set $\b X^n = (\b X^{1,n},...,\b X^{N,n})$ by $\b X ^{j,n} := X^{j,n}$ if $j\ne i$ and $\b X^{i,n} := \tilde X^{i,n}$.
By optimality of $\eta^n$ we have
\begin{align*}
\mathbb E \bigg[ \int_0^T    X_t^{n} Q ^i_t X_t^{n}  dt  & +   X_T^{n} Q ^i_T X_T^{n}  + \int_{[0,T]} (c^{i,+}_t d\xi^{i,n}_t + c^{i,-}_t d\zeta ^{i,n}_t)  \bigg] \\
& \leq J^i(\eta^{i,n}, \eta^{-i,n})  \leq J^i(0, \eta^{-i,n}) \\
& = \mathbb E \bigg[ \int_0^T   \b X_t^{n} Q ^i_t \b X_t^{n} dt  + \b  X_T^{n} Q ^i_T \b X_T^{n} \bigg]. 
\end{align*}
Now, using the fact that $c^{i,+}_t, c^{i,-}_t \geq \b c >0$ (see in Condition \ref{condition c} in Assumption \ref{assumption}), by employing H\"older inequality with exponent $2$ we find
\begin{align*}
    \b c \mathbb E \Big[ \xi^{i,n}_T + \zeta^{i,n}_T  \Big] 
    & \leq \mathbb E \bigg[   \int_{[0,T]} (c^{i,+}_t d\xi^{i,n,+}_t + c^{i,-}_t d\xi^{i,n,-}_t)  \bigg] \\
    & \leq \mathbb E \bigg[ \int_0^T \Big(   \b X_t^{n} Q ^i_t \b X_t^{n} -  X_t^{n} Q ^i_t X_t^{n}\Big) dt  + \b  X_T^{n} Q ^i_T \b X_T^{n} - X_T^{n} Q ^i_T X_T^{n} \bigg] \\
    & \leq C \bigg( 1+ \mathbb E \bigg[ \int_0^T |X_t^{n}|^2 dt + | X_T^{n}|^2 \bigg] \bigg), 
\end{align*}
so that, thanks to the estimates in \eqref{eq a priori estimates int X}, we obtain
$$
 \sup_n \mathbb E [ \xi^{i,n}_T + \zeta^{i,n}_T] < \infty. 
$$
Finally, since $i \in \{1,...,N\}$ is arbitrary, we obtain
$$
 \sup_n \mathbb E [ |\xi_T^n| + |\zeta^n_T|] \leq \sup_n \sum_{i=1}^N \mathbb E [ \xi^{i,n}_T +\zeta^{i,n}_T] < \infty,
$$
and, by classical Gr\"onwall estimate, we conclude that
$$
\sup_n \mathbb E \bigg[ \sup_{t\in [0,T]} |X_t^n| \bigg] < \infty, 
$$
completing the proof.
\end{proof}

We are now ready to identify the accumulation points of the sequence of Nash equilibria $(X^n,Y^n,\xi^n,\zeta^n)_n$ of the $n$-Lipschitz game.
To this end, for a generic $d \in \mathbb N$, introduce the Hilbert space $\mathbb H _T^{2,d}$ with norm $\| \cdot \|_T^{2,d} $ defined as
\begin{equation*}
    \mathbb H _T^{2,d}:= \{ M \in \mathbb H ^{2,d} \text{ s.t. } \| M \|_T^{2,d} < \infty \} 
    \quad \text{and} \quad
    \| M \|_T^{2,d} :  = \mathbb E \bigg[ \int_0^T |M_t|^2 dt + |M_T|^2 \bigg],
\end{equation*}
and set $\mathbb H ^{2}_T := \mathbb H ^{2,1}_T$. 
Also, on $ \mathbb H _T^{2,d}$ we can consider the weak convergence; that is, for $M,\, M^n \in  \mathbb H _T^{2,d}$, $n\in \mathbb N$, we say that 
\begin{equation*}
    M^n \rightarrow M \text{ as $n\to \infty$, weakly in $\mathbb H _T^{2,d}$,}
\end{equation*}
if, for any $H \in \mathbb H _T^{2,d}$, one has
\begin{align*}
   \lim_n \E \bigg[ \int_0^T H_t M_t^{n} dt +  H_T M_T^{n} \bigg]  = \E \bigg[ \int_0^T H_t M_t dt +  H_T M_T  \bigg].  
\end{align*}

We now state the following convergence result.
\begin{proposition}
    \label{prop weak convergence} 
There exists a subsequence of $(X^n, Y^n, \xi^n, \zeta^n)_n$ (still indexed by
$n$) and processes $(X, Y) = (X^1, ..., X^N, Y^1,...,Y^N) \in \mathbb H _T^{2,2N}$ and $( \xi, \zeta) = (\xi^1,...,\xi^N, \zeta^1,...,\zeta^N) \in \tilde{\mathcal A }^{2N}$ such that:
\begin{enumerate}
\item \label{eq weak convergence x Y} $(X^n, Y^n)_n \to (X, Y)$ as $n\to \infty$, weakly in $\mathbb H _T^{2,2N}$;
\item\label{eq conv xi zeta in distribution}  For $\bar \xi^{i,m} := \frac1n \sum_{n=1}^m \xi^{i,n}$ and $\bar \zeta^{i,m} := \frac1n \sum_{n=1}^m \zeta^{i,n}$, for
$\P$-a.a. $\omega \in  \Omega$, the convergence
\begin{align*}
\bar \xi_t^{i,m}(\omega)  \rightarrow \xi_t^{i} (\omega)  \text{ for any continuity point of $\xi^i(\omega)$ and } 
\bar \xi_T^{i,m}(\omega)  \rightarrow \xi_T^{i}(\omega), \\ \notag
\bar \zeta_t^{i,m}(\omega)  \rightarrow \zeta_t^{i} (\omega)  \text{ for any continuity point of $\zeta^i(\omega)$ and } 
\bar \zeta_T^{i,m}(\omega)  \rightarrow \zeta_T^{i}(\omega),  \notag 
\end{align*} 
as $m \to \infty$ holds, for any $i=1,...,N$;
\item The profile strategy $(\xi, \zeta)$ is admissible. 
\end{enumerate} 
\end{proposition} 
\begin{proof}
Since the process $Y^{i,n}$ solves the BSDE in \eqref{eq FBSDE n}, we have
$$
\E \big[ |Y^{i,n}_t|^2 \big] \leq C \E \bigg[ \int_0^T |X_t^{n}|^2 dt + | X_T^{n}|^2 \bigg],
$$
so that, by Lemma \ref{lemma a priori estimates}, we have
$$
\sup_n \E \bigg[ \int_0^T |Y_t^{n}|^2 dt + | Y_T^{n}|^2 \bigg] < \infty.
$$
The latter, together with the estimates in Lemma \ref{lemma a priori estimates}, allows to find a subsequence of $(X^n, Y^n)_n$ (still labelled by $n$) and a process $(X,Y) = (X^1, ..., X^N, Y^1,...,Y^N)$ such that
$(X^n, Y^n)$ converges to $ (X,Y)$ as $n\to \infty$, weakly in $\mathbb H _T^{2,2 N}$.

We next identify the limits for the sequence $(\xi^n,\zeta^n)_n$. 
By the estimates in Lemma \ref{lemma a priori estimates}, thanks to Lemma 3.5 in \cite{K} we can find processes $\xi = (\xi^1,...,\xi^N) \in \tilde \A ^N$ and $\zeta = (\zeta^1,...,\zeta^N) \in \tilde \A ^N$ and a subsequence of indexes (not relabelled) such that, for any further subsequence, by setting
\begin{equation*}
\bar \xi^{i,m} := \frac1n \sum_{n=1}^m \xi^{i,n}
\quad \text{and} \quad 
\bar \zeta^{i,m} := \frac1n \sum_{n=1}^m \zeta^{i,n},
\end{equation*}
we have, for $\P$-a.a. $\omega \in  \Omega$, the convergence
\begin{align}\label{eq conv xi zeta in distribution in proof}
\bar \xi_t^{i,m}(\omega)  \rightarrow \xi_t^{i} (\omega)  \text{ for any continuity point of $\xi^i(\omega)$ and } 
\bar \xi_T^{i,m}(\omega)  \rightarrow \xi_T^{i}(\omega), \\ \notag
\bar \zeta_t^{i,m}(\omega)  \rightarrow \zeta_t^{i} (\omega)  \text{ for any continuity point of $\zeta^i(\omega)$ and } 
\bar \zeta_T^{i,m}(\omega)  \rightarrow \zeta_T^{i}(\omega),  \notag 
\end{align} 
as $m \to \infty$ for any $i=1,...,N$.
On the other hand,  for any $i=1,...,N$, we can define the processes $v^{i,n}:= \xi^{i,n}-\zeta^{i,n}$ and, by
Lemma \ref{lemma a priori estimates}, we have
\begin{align*}
    \sup_n \E \bigg [ \int_0^T |v^{i,n}_t|^2 dt + |v^{i,n}_T|^2 \bigg] \leq  C \sup_n \E \bigg[ \int_0^T |X_t^{i,n}|^2 dt + | X_T^{i,n}|^2 \bigg] < \infty. 
\end{align*}
Thus, there exists a further subsequence (again, not relabelled) and a process $v^i \in \mathbb H ^{2}_T $ such that
\begin{equation}\label{eq weak convergence v}
      v^{i,n} \rightarrow v^i \text{ as $n\to \infty$, weakly in $\mathbb H _T^{2}$.}
\end{equation}
Moreover, by Banach-Saks theorem, we can find another subsequence of $(v^{i,n})_n$ (still labelled by $n$) such that 
$\bar v^{i,m} := \frac1n \sum_{n=1}^m v^{i,n} \to v^i$, as $m \to \infty$, strongly in $\mathbb H ^{2}_T$. 
Thus, up to a subsequence (still labelled by $m$), we have the convergence
$$
\bar v _t^{i,m}  \to v^i_t, \text{ as $m\to \infty$, $\P \otimes dt$-a.e.\ in $\Omega \times [0,T]$.}
$$
The latter limit, together with \eqref{eq conv xi zeta in distribution in proof}, implies that $v^i = \xi^i - \zeta^i$.

Finally, since $v^i \in \mathbb H ^{2}_T$, we conclude that $(\xi^i,\zeta^i)$ is admissible, completing the proof of the proposition.
\end{proof}

\subsubsection{Properties of limit points}

In the next two proposition we will show that the accumulation point $(X,Y,\xi,\zeta)$ satisfies the conditions of Theorem \eqref{thm SMP}.

\begin{proposition}
    \label{prop weak sol FBSDE}
    The process $(X,Y,\xi,\zeta)$ solves the FBSDE 
\begin{equation}\label{eq FBSDE limit}
\begin{cases}   
d  X^{i}_t & = \big( a^i_t + b^i_t X ^{i}_t \big) dt  + \sigma ^i_t dW^i_t  + d  \xi^{i}_t - d \zeta^{i}_t \quad X_{0-}^{i}=x^i_0, \quad i=1,...,N,\\
Y_t^{i} &= 2 \E \Big[ \Gamma^i_{t,T} Q^{i;i}_T X_T + \int_t^T \Gamma^i_{t,s}  Q^{i;i}_s X_s ds    \Big| \mathcal F _t \Big], \quad i=1,...,N.
\end{cases}
\end{equation}
\end{proposition}

\begin{proof}
Take $i \in \{1,...,N\}$.
We first prove that $X^i$ solves the forward equation. 
Since, for any $n$, the process $X^{i,n}$ solves the forward equation in \eqref{eq FBSDE n}, we have
$$
X_t^{i,n} = A_t^{i,n}  + \int_0^t  b_s^i X^{i,n}_s ds + v_t^{i,n},  \quad A_t^{i,n} := x_0^i+ \int_0^t a_s^i ds + \int_0^t \sigma_s^{i,n} dW_s^i.
$$
Then, for any $M \in \mathbb H ^2$, via an integration by parts we obtain
\begin{align*}
\E \bigg[ \int_0^T M_t X^{i,n}_t dt \bigg]
& = \E \bigg[ \int_0^T M_t \Big(A_t^{i,n} + \int_0^t  b_s^i X^{i,n}_s ds + v_t^{i,n} \Big) dt \bigg] \\
& = \E \bigg[ \int_0^T M_t (A_t^{i,n} + v_t^{i,n})dt \bigg] \\
& \quad +\E \bigg[  \int_0^T  \Big( \int_0^T M_s ds \Big)  b_t^i X^{i,n}_t dt -  \int_0^T  \Big( \int_0^t M_s ds \Big) b_t^i X^{i,n}_t dt \bigg].
\end{align*}
Notice that, from the definition of $\sigma^{i,n}$ in \eqref{eq ellipticity n}, we have 
$$
\lim_n \mathbb E \bigg[ \int_0^T |A^{i,n}_t - A^i_t|^2 dt \bigg] = 0, 
\quad\text{ where } \quad 
 A_t^{i} := x_0^i+ \int_0^t a_s^i ds + \int_0^t \sigma_s^{i} dW_s^i.
$$
Hence, the convergence established in Proposition \ref{prop weak convergence} (see also \eqref{eq weak convergence v}) allows to take limits as $n\to \infty$ in the latter equality, and integrating again by parts, we conclude that
$$
\E \bigg[ \int_0^T M_t X^{i}_t dt \bigg]   = \E \bigg[ \int_0^T M_t \Big(A_t^i + \int_0^t  b_s^i X^{i}_s ds + v_t^{i} \Big) dt \bigg], \quad \text{for any $M \in \mathbb H^2$}.
$$
Thus, the forward equation $X_t^{i} = A_t^i + \int_0^t  b_s^i X^{i}_s ds + v_t^{i}$ holds $\mathbb P \otimes dt $-a.e.
 
We now prove that $Y^i$ solves the backward equation.
Since, for any $n$, the process $Y^{i,n}$ solves the backward equation in \eqref{eq FBSDE n}, we have, for any $M \in \mathbb H _T^2$, the identity 
$$
\E \bigg[ \int_0^T M_t Y^{i,n}_t dt \bigg]   = 2 \E \bigg[ \int_0^T M_t \Big( \E \Big[ \Gamma^i_{t,T} Q^{i;i}_T X^n_t + \int_t^T \Gamma^i_{t,s}  Q^{i;i}_s X^n_s ds  \Big| \mathcal F _t \Big] \Big) dt \bigg], 
$$
and, by  using Theorem 1.33 in \cite{jacod1979calcul} and then an integration by parts,  we obtain 
\begin{align*}
\E \bigg[ \int_0^T M_t Y^{i,n}_t dt \bigg]   
&= 2 \E \bigg[ \int_0^T M_t \Big( \Gamma^i_{t,T} Q^{i;i}_T X^n_T + \int_t^T \Gamma^i_{t,s}  Q^{i;i}_s X^n_s ds   \Big) dt \bigg] \\
&= 2 \E \bigg[ X^n_T Q^{i;i}_T \int_0^T M_t  \Gamma^i_{t,T} dt  + \int_0^T \Big( \int_0^t M_s \Gamma^i_{s,t} ds \Big)  Q^{i;i}_t X^n_t  dt \bigg].
\end{align*}
Finally, thanks to the convergence established in Proposition \ref{prop weak convergence}, we can take limits as $n\to \infty$ in the latter equality and, using the same steps backward, we conclude that
$$
\E \bigg[ \int_0^T M_t Y^{i}_t dt \bigg]   = 2 \E \bigg[ \int_0^T M_t \Big( \E \Big[ \Gamma^i_{t,T} Q^{i;i}_T X_T + \int_t^T \Gamma^i_{t,s}  Q^{i;i}_s X_s ds  \Big| \mathcal F _t \Big] \Big) dt \bigg], 
$$
for any $M \in \mathbb H _T^2$, so that $Y^i$ solves the backward equation.
\end{proof}

\begin{proposition}
    \label{prop sufficient conditions}
For every $i=1,...,N$, the following conditions hold true: 
\begin{enumerate}
\item\label{claim inside waiting}$ Y ^{i}_t + c_t^{i,+} \geq 0$ and $-  Y ^{i}_t + c_t^{i,-} \geq 0$, for any $t \in [0,T], \ \mathbb P$-a.s.;
\item\label{claim activation of control} $\int_{[0,T]} ( Y ^{i}_t + c_t^{i,+}) d \xi ^i_t = 0$ and $\int_{[0,T]} (- Y ^{i}_t + c_t^{i,-}) d \zeta ^i_t =0$, $\mathbb P $-a.s.
\end{enumerate}  
\end{proposition}

\begin{proof}
We prove each claim separately. 
\smallbreak\noindent
\emph{Proof of \ref{claim inside waiting}.}
By Lemma \ref{lemma subgradient identities}, we have 
\begin{align}
 J^i( 0,0 ; \eta ^{-i,n}) - J^i(\xi^{i,n}&, \zeta^{i,n}; \eta ^{-i,n} )   \\ \notag 
& \geq - \mathbb E \bigg[ \int_0^T (Y_t^{i,n} + c_t^{i,+}) d\xi^{i,n}_t +\int_0^T (- Y_t^{i,n} + c_t^{i,-}) d\zeta^{i,n}_t
\bigg],
\end{align}
where, in the last equality, we have used the integrability of $\xi^{i,n}$ and $\zeta^{i,n}$.
Next, for $y \in \mathbb R$, set $y^+ := \max\{y,0\}$ and $y^- := \max\{ - y, 0 \}$.  
By using the necessary conditions in \eqref{eq necessary cond n}, we obtain
\begin{align*}
  n \mathbb E \bigg[ \int_0^T (Y_t^{i,n} + c_t^{i,+})^- dt &+  \int_0^T (- Y_t^{i,n} + c_t^{i,-})^- dt \bigg] \\
  & \leq   J^i( 0 ; \eta ^{-i,n}) - J^i(\xi^{i,n}, \zeta^{i,n}; \eta ^{-i,n} ), 
\end{align*}
so that, thanks to the boundedness of $Q^i$, we have
\begin{align*}
    n \mathbb E \bigg[ \int_0^T (Y_t^{i,n} + c_t^{i,+})^- dt &+ \int_0^T (- Y_t^{i,n} + c_t^{i,-})^- dt \bigg] \\ 
    &\leq  C\bigg( 1+ \mathbb E \bigg[ \int_0^T |X_t^{n}|^2 dt + | X_T^{n}|^2 \bigg] \bigg).
\end{align*}
Hence, by Lemma \ref{lemma a priori estimates} we deduce that
\begin{equation}
    \label{eq limit FOCS 1}
    \lim_n \mathbb E \bigg[ \int_0^T (Y_t^{i,n} + c_t^{i,+})^- dt \bigg] =  \lim_n \mathbb E \bigg[ \int_0^T (- Y_t^{i,n} + c_t^{i,-})^- dt \bigg] = 0. 
\end{equation}
From the latter equality, using that $Y^{i,n}$ converges weakly to $Y^i$ as $n\to \infty$ (cf.\ Proposition \ref{prop weak convergence}),  we deduce that
 \begin{align*}
     0 \leq \E \bigg[ \int_0^T (Y_t^{i} + c_t^{i,+})^- dt \bigg]  
     & = - \lim_n \E \bigg[ \int_0^T (Y_t^{i,n} + c_t^{i,+}) \mathds 1 _{\{ Y_t^{i} + c_t^{i,+} \leq 0\} } dt \bigg] \\
& \leq \lim_n \E \bigg[ \int_0^T (Y_t^{i,n} + c_t^{i,+})^- \mathds 1 _{\{ Y_t^{i} + c_t^{i,+} \leq 0\} } dt \bigg] \\
       & \leq  \lim_n \E \bigg[ \int_0^T (Y_t^{i,n} + c_t^{i,+})^-  dt \bigg] = 0.
 \end{align*}
Similarly, we find
 $$
 \E \bigg[ \int_0^T ( - Y_t^{i} + c_t^{i,-})^- dt \bigg] = 0,
 $$
completing the proof of Claim \ref{claim inside waiting}.

\smallbreak\noindent
\emph{Proof of \ref{claim activation of control}}.
Using \eqref{eq key idenity} (see the proof of Lemma \ref{lemma subgradient identities}) 
 with $\eta = \eta^{-i,n}$ and $(\b \xi, \b \zeta) = (0,0)$, denoting by $\tilde X ^{i,n}$ the solution to \eqref{eq SDE n uncontrolled}, we obtain
\begin{align*}  
\E \bigg[ \int_{[0,T]} Y^{i,n}_t    d (\xi^{i,n} - \zeta^{i,n}) _t \bigg] 
&= 2 \E \bigg[ \int_0^T \Big( \sum _{j=1}^N q_t^{i,j;i} X_t^{j,n}\Big)(X^{i,n}_t -\tilde X ^{i,n}_t) dt  \bigg] \\
& \quad + 2 \E \bigg[ \Big( \sum _{j=1}^N q_T^{i,j;i} X_T^{j,n}\Big)(X^{i,n}_T -\tilde X ^{i,n}_T)\bigg].\notag 
\end{align*}
Thus, summing over $i=1,...,N$, for $\hat Q$ as in \eqref{eq def matrix bar Q}  and $\tilde X ^n :=(\tilde X^{1,n}, ...,\tilde X ^{N,n})$ we find
\begin{align}\label{eq FOCs 2: repr n}
\sum_{i=1}^N \E \bigg[ \int_{[0,T]} Y^{i,n}_t    d (\xi^{i,n} - \zeta^{i,n}) _t \bigg] 
& = \E \bigg[ \int_0^T  ( X^n_t \hat Q _t X^n_t -\tilde  X^n_t  \hat Q _t  X _t^n) dt \bigg] \\ \notag
& \quad+  \E \big[ X^n_T \hat Q _T X^n_T -\tilde X^n_T  \hat Q _T  X _T^n \big]. \notag 
\end{align}
Similarly, for $\tilde X = (\tilde X ^1, ..., \tilde X ^N )$, with $\tilde X ^j$ solution to the SDE \eqref{eq SDE uncontrolled} for $j=1,...,N$, we obtain the identity 
\begin{align}\label{eq FOCs 2: repr} 
\sum_{i=1}^N \E \bigg[ \int_{[0,T]} Y^{i}_t    d (\xi^{i} - \zeta^{i}) _t \bigg] 
& = \E \bigg[ \int_0^T  ( X_t \hat Q _t X_t -  \tilde X_t  \hat Q _t X _t) dt \bigg] \\ \notag
& \quad+  \E \big[ X_T \hat Q _T X_T - \tilde X_T  \hat Q _T  X _T  \big].\notag 
\end{align}
 
From \eqref{eq ellipticity n}, we notice that
$$
\lim_n \mathbb E \bigg[ \int_0^T |\tilde X^n_t -\tilde X _t|^2 dt \bigg] = 0.
$$
Therefore, since $X^n \to X$ weakly (see Proposition \ref{prop weak convergence}),  we have
$$
\E \bigg[ \int_0^T \tilde X_t \hat Q _t  X _t dt + \tilde X_T \hat Q _T X _T \bigg] 
= \lim_n \E \bigg[ \int_0^T \tilde X_t^n \hat Q _t  X_t^n dt +\tilde X_T^n \hat Q _T  X _T^n \bigg]
$$
and, by convexity of the map $x \mapsto x \hat Q x$ (cf.\ Condition \ref{condition Q bar Q >0} in Assumption \ref{assumption}), we find
$$
\E \bigg[ \int_0^T X_t \hat Q _t X_t dt + X_T \hat Q _T  X _T \bigg] \leq \liminf_n \E \bigg[ \int_0^T X_t^n \hat Q _t X_t^n dt + X_T^n \hat Q _T  X _T^n\bigg].
$$
Hence, by using the latter limits in \eqref{eq FOCs 2: repr n} and \eqref{eq FOCs 2: repr}, we obtain
\begin{align}\label{eq liminf first}
\sum_{i=1}^N \E \bigg[ \int_{[0,T]} Y^{i}_t  d  ( \xi^{i} - \zeta^{i} )_t \bigg] \leq \liminf_n\sum_{i=1}^N \E \bigg[ \int_{[0,T]} Y^{i,n}_t d( \xi^{i,n} - \zeta^{i,n}) _t \bigg],
\end{align}

Next, for a suitable subsequence of indexes $(n_k)_k$, since the functions $c^{i,+}$ and $c^{i,-}$ are bounded and continuous, the limits at Point \ref{eq conv xi zeta in distribution} in Lemma \ref{prop weak convergence} give
\begin{equation}\label{eq lim cesaro c d xi}
\sum_{i=1}^N \E \bigg[ \int_{[0,T]} \Big( c_t^{i,+} d \xi^{i}_t  + c_t^{i,-} d \zeta^{i}_t \Big) \bigg] 
= \lim_m \frac1m \sum_{k=1}^m \sum_{i=1}^N \E \bigg[ \int_{[0,T]} \Big( c_t^{i,+} d \xi^{i, n_k}_t  + c_t^{i,-} d \zeta^{i,n_k}_t \Big) \bigg].
\end{equation}
Moreover, by the limits in \eqref{eq liminf first}, we also have
\begin{align}\label{eq liminf second}
\sum_{i=1}^N \E \bigg[ \int_{[0,T]} Y^{i}_t  d  ( \xi^{i} - \zeta^{i} )_t \bigg] 
\leq \liminf_m \frac1m \sum_{k=1}^m \sum_{i=1}^N \E \bigg[ \int_{[0,T]} Y^{i,n_k}_t d( \xi^{i,n_k} - \zeta^{i,n_k}) _t \bigg].
\end{align}

Finally, 
by the step 1 in this proof, the integrals $\int_{[0,T]} (Y^{i}_t + c_t^{i,+})  d \xi^{i}_t  $ and $ \int_{[0,T]} ( -Y^{i}_t + c_t^{i,-}) d \zeta^{i}_t$ are well defined (possibly equal to $+\infty$)
and, from \eqref{eq lim cesaro c d xi} and \eqref{eq liminf second}, we conclude that 
\begin{align*}
\sum_{i=1}^N & \E \bigg[ \int_{[0,T]} (Y^{i}_t + c_t^{i,+}) d \xi^{i}_t  + \int_{[0,T]} ( -Y^{i}_t + c_t^{i,-}) d \zeta^{i}_t \bigg] \\
& \leq  \liminf_m \frac1m \sum_{k=1}^m \sum_{i=1}^N \E \bigg[ \int_{[0,T]} (Y^{i,n_k}_t + c_t^{i,+}) d \xi^{i,n_k}_t  + \int_{[0,T]} ( -Y^{i,n_k}_t + c_t^{i,-}) d \zeta^{i,n_k}_t \bigg],\\
& =  \liminf_m \frac1m \sum_{k=1}^m \sum_{i=1}^N \bigg(- n_k \E \bigg[ \int_0^T \Big( (Y^{i,n_k}_t + c_t^{i,+})^-   +  ( -Y^{i,n_k}_t + c_t^{i,-})^- \Big) dt \bigg] \bigg) \leq 0,
\end{align*}
where we have used the necessary conditions in \eqref{eq necessary cond n}.
The latter inequality, combined with Claim \ref{claim inside waiting}, in turn implies that
$$
\int_{[0,T]} (Y^{i}_t + c_t^{i,+}) d \xi^{i}_t  = \int_{[0,T]} ( -Y^{i}_t + c_t^{i,-}) d \zeta^{i}_t= 0,
$$
thus completing the proof of the proposition.
\end{proof}

In order to conclude the proof of Theorem \ref{thm main}, we only remain to observe that, by Propositions \ref{prop weak sol FBSDE} and \ref{prop sufficient conditions}, the constructed $\eta$  satisfies the conditions of Theorem \ref{thm SMP}, and it is therefore a Nash equilibrium.

\section{An application to oligopoly investment games}
    \label{subsection oligopoly investment games}
We consider $N$ firms competing in a market by producing and selling a certain good.
The stochastic demand of such a good is modeled by the one-dimensional diffusion process
$$
d X^0_t = \mu^0(t,X^0_t)   dt + \sigma^0(t,X^0_t) dW^0_t, \quad X_0^0 =x^0_0 \in \R,
$$
which is driven by a one-dimensional Brownian motion $W^0$ on some complete probability space $(\Omega, \mathcal F, \mathbb P)$.
The functions $\mu^0, \sigma^0 :[0,T] \times \R \to \R$ are assumed to be Lipschitz continuous and the diffusive term $\sigma^0$ is assumed to satisfy the nondegeneracy condition
$$
0 < \underline \sigma \leq \sigma^0(t,x) \leq \bar \sigma, \quad \text{for any $t \in [0,T]$, $x\in \R$, for some $\underline \sigma, \bar \sigma \in \R$.}
$$

Following the fluctuations of the demand $X^0$, each company $i$ can  expand its capital stock $X^i$ through an irreversible investment strategy $\xi^i$. 
Since $\sigma^0$ is nondegenerate, the filtration generated by $X^0$ coincides with the filtration generated by $W^0$.
Let $\mathbb F ^0 = ( \cF _t^0)_t$ the right-continuous extension of the filtration generated by $W^0$, augmented by the $\P$-null sets.
Thus, strategies $\xi^i$ are $\mathbb F ^0$-adapted, nonnegative, nondecreasing, c\`adl\`ag processes with $\E [\xi_T^i ]< \infty$, and  the capital stock of firm $i$ evolves as 
$$
dX_t^i  =  - \delta^i X_t^i dt  + d \xi_t^i, \quad X_{0-}^i = x^i_0 \geq 0,
$$ 
where the parameter $\delta^i>0$ measures the natural deterioration of the capital.
The production output of firm $i$ is given by the multiple $\alpha^i X_t^i$, for some $\alpha^i >0$.

Assuming a linear demand (as in  \cite{Back&Paulsen}, at the end of Section 1), the price at time $t$ of the good is given by
$
 X^0_t - \gamma \sum_{j=1}^N \alpha^j X_t^j,
$
for a parameter $\gamma >0$.
Hence, the profits from sales of company $i$ is given by
$
X_t^{i} \big( X_t^0 - \gamma \sum_{j=1}^N \alpha^j X_t^{j} \big).
$
Moreover, we assume that the cost faced by company $i$ per unit of investment is $c^i>0$, and that the company's discount factor is $\rho > 0$.
Summarizing, each company aims at maximizing the net quadratic-singular profit functional
\begin{equation*}
P^i(\xi^i,\xi^{-i}):= \mathbb{E} \bigg[ \int_0^T e^{-\rho t} \alpha^i X_t^{i} \Big( X_t^0 - \gamma \sum_{j=1}^N  \alpha^j X_t^{j} \Big) dt - c^i \int_{[0,T]} e^{-\rho t} d\xi^{i}_{t}  \bigg].
\end{equation*}

Slightly adapting Theorem \ref{thm main}, we can show existence of equilibrium investment strategies.
\begin{theorem}
There exists an $\mathbb F ^0$-adapted Nash equilibrium.  
\end{theorem}

\begin{proof}
    In order to simplify the notation we take $\alpha^1=...=\alpha^N=1$ and $\rho=0$ (the proof in the general case is analogous). The rest of the proof is dived in two steps. 
    \smallbreak\noindent
    \emph{Step 1.} We first give a sketch of how to construct a Nash equilibrium as in Theorem \ref{thm main}.
    Without loss of generality, we can assume the probability space $(\Omega, \mathcal F, \mathbb P)$ to be large enough to accommodate $N$ independent Brownian motions $W^i$, $i=1,...,N$, which are independent from $W^0$. 
    Let $\mathbb F  = ( \cF _t)_t$ the right-continuous extension of the filtration generated by $(W^0,W^1,...,W^N)$, augmented by the $\P$-null sets.  

    We observe that the symmetry of $Q^i$ and the non-degeneracy of $\bar Q$ in Assumption \ref{assumption} are not satisfied. 
    However, despite the presence of the extra uncontrolled dynamics $X^0$, Theorem \ref{thm main} applies (with mimimal adjustment) and provides the existence of a Nash equilibrium $\xi = (\xi^1, ..., \xi^N)$ which is $\mathbb F$-adapted, with $\E [ |\xi_T|^2 ] < \infty$.
    In particular, the main difference is in the estimates of  Lemma \ref{lemma a priori estimates}, which can be recovered as follows. 
    For $n\in \mathbb N$, let $\xi^n$ be a Nash equilibrium of the related $n$-Lipschitz game,
     and denote by $X^n= (X^{1,n},...,X^{N,n})$ and $\tilde X^n =(\tilde X ^{1,n}, ..., \tilde X ^{N,n}) $ the solutions to the controlled and uncontrolled equations 
     \begin{align*}
         d X _t^{i,n} & =  - \delta^i  X _t^{i,n} dt + \frac1n dW^i_t +d \xi^{i,n}_t, \quad  X _{0 -}^{i,n} = x^i_0, \\ 
     d \tilde X _t^{i,n} & =  - \delta^i \tilde X _t^{i,n} dt + \frac1n dW^i_t , \quad \tilde X _{0}^{i,n} = x^i_0,
     \end{align*}
     respectively, for $i=1,...,N$.
    By optimality, for $i=1,...,N$ we have $ P^i(\xi^{i,n},\xi^{-i,n}) \geq P^i(0,\xi^{-i,n})$, which implies that
    \begin{align*}
        \mathbb{E} \bigg[ \int_0^T  X_t^{i,n} \Big(   \gamma \sum_{j=1}^N   X_t^{j,n} - X_t^{0} \Big) dt \bigg] \leq \mathbb{E} \bigg[ \int_0^T  \tilde X_t^{i,n} \Big(  \gamma \tilde X_t^{i,n}+ \gamma \sum_{j\ne i}^N   X_t^{j} - X_t^{0} \Big) dt \bigg].
    \end{align*}
    Thus, summing over $i$, we obtain
    \begin{align*}
        \mathbb{E} \bigg[ \int_0^T     \Big( \sum_{j=1}^N   X_t^{j,n} \Big) ^2  dt \bigg] 
        \leq 
        C \bigg( 1+\mathbb{E} \bigg[ \int_0^T  \sum_{i=1}^N   \Big( X_t^{0}  X_t^{i,n} + \gamma \tilde X_t^{i,n}   \sum_{j\ne i}^N   X_t^{j} \Big) dt \bigg] \bigg),
    \end{align*}
    and, using that $X_t^{i,n} >0$ for any $t \in [0,T]$, $\mathbb P $-a.s., thanks to H\"older inequality we find
    \begin{align*}
        \mathbb{E} \bigg[ \int_0^T  |X_t^{n}|^2  dt \bigg] 
        \leq 
        C \bigg( 1+ \bigg( \mathbb{E} \bigg[ \int_0^T  |X_t^{n}|^2  dt \bigg]\bigg)^{1/2} \bigg).
    \end{align*}
    Hence, we conclude that $\sup_n \mathbb{E} [ \int_0^T  |X_t^{n}|^2  dt ] < \infty$ and (as in Lemma \ref{lemma a priori estimates}) that  $\sup_n \mathbb{E} [ |\xi_T^n| ] < \infty$.
    
    We also underline that there is a difference in the optimality conditions of Theorem \eqref{thm SMP}.
    Indeed, if the process $(X,Y) = (X^1,...,X^N,Y^1,...,Y^N) \in\mathbb H^{2,2N }$ is associated to a Nash equilibrium $\xi$, then it solves the FBSDE
    $$
    \begin{cases}
        X_t^i = x_0^i -\delta^i \int_0^t X_s^ids + \xi^i_t, \\
        Y_t^{i} = \E \big[  \int_t^T  e^{-\delta^i (s-t)} (\gamma \sum_{j=1}^N X_s^j +\gamma X_s^i - X_s^0 )  ds    \big| \mathcal F _t \big],  
    \end{cases}
    $$
    and, by the analogus of  Theorem \ref{thm SMP} in the current setting, the equilibrium $\xi$ satisfies the conditions:    
    \begin{align}
        \label{cond 1 example} &\text{$Y ^{i}_t + c^{i}  \geq 0$,  for any $t \in [0,T], \ \mathbb P$-a.s.;} \\
        \label{cond 2 example} &\text{$\int_{[0,T]} ( Y ^{i}_t + c^{i}) d \xi ^i_t = 0$, $\mathbb P $-a.s.}
    \end{align}
    
    \smallbreak\noindent
    \emph{Step 2.} 
    We now construct an $\mathbb F ^0$-adapted equilibrium.
    Set $\bar \xi : = (\bar \xi ^1, ..., \bar \xi ^N)$, where 
    $$
    \bar \xi ^i : = (\E [ \xi _t^i |\cF _t^0 ])_t, \quad \text{for}\quad i=1,...,N.
    $$
    Clearly, the processes $\bar \xi ^i$ are $\mathbb F ^0$-adapted and, since $\E [ \xi _t^i |\cF _t^0 ] = \E [ \xi _t^i |\cF _T^0 ] \ \mathbb P$-a.s., we see that $\bar \xi ^i$ are nondecreasing and c\`adl\`ag. 
    Next, set
    $$
    \bar X ^i : = (\E [ X _t^i |\cF _t^0 ])_t \quad \text{and} \quad  \bar Y ^i : = (\E [ Y _t^i |\cF _t^0 ])_t, \quad \text{for} \quad i=1,...,N.
    $$ 
    With elementary arguments, we find 
    $$
    \begin{cases}
        \bar X _t^i = x_0^i -\delta^i \int_0^t \bar X _s^ids + \bar \xi ^i_t, \\
        \bar Y _t^{i} = \E \big[  \int_t^T  e^{-\delta^i (s-t)} (\gamma \sum_{j=1}^N \bar X_s^j +\gamma \bar X_s^i -  X_s^0 )  ds    \big| \mathcal F _t^0 \big].
    \end{cases}
    $$
    Next, we want to show that $\bar \xi$ is a Nash equilibrium by checking the sufficient conditions of Theorem \ref{thm SMP}. 
    By taking the conditional expectation in \eqref{cond 1 example}, we find
    \begin{equation}\label{eq condition 1 bar} 
    \text{$\bar Y ^{i}_t + c^{i} \geq 0$, for any $t \in [0,T], \ \mathbb P$-a.s.}
    \end{equation} 
    Also,  similarly to  \eqref{eq key idenity}, denoting by $\tilde X^i $ the solution to the uncontrolled equation $d \tilde X _t^i  =  - \delta^i \tilde X _t^i dt , \ \tilde X _{0}^i = x^i_0$, we have
    \begin{align*}
      \mathbb E \bigg[ \int_{[0,T]} ( \bar Y ^{i}_t + c^{i}) d \bar \xi ^i_t \bigg] 
      & = \E \bigg[ \int_0^T \Big(\gamma \sum_{j=1}^N \bar X_t^j +\gamma \bar X_t^i -  X_t^0 \Big)  ( \bar X ^i_t - \tilde X^i_t) dt + c^{i} \bar \xi ^i_T \bigg]. 
    \end{align*} 
    Moreover, noticing that $X^0$ is $\mathbb F^0$-adapted and that $\tilde X^i$ is deterministic, summing over $i$ and using Jensen inequality for conditional expectation we obtain 
    \begin{align*}
     \sum_{i=1}^N \mathbb E \bigg[ \int_{[0,T]} ( \bar Y ^{i}_t + c^{i}) d \bar \xi ^i_t \bigg] 
      & = \E \bigg[ \int_0^T \bigg( \gamma \Big(\sum_{i=1}^N \bar X_t^i \Big) ^2 + \gamma \sum_{i=1}^N \big( \bar X_t^i \big)^2 -  X_t^0 \sum_{i=1}^N  \bar X_t^i \bigg)  dt \\
      & \quad  -  \sum_{i=1}^N \int_0^T \Big(\gamma \sum_{j=1}^N \bar X_t^j + \gamma \bar X_t^i -  X_t^0 \Big)   \tilde X^i_t dt +  \sum_{i=1}^N c^{i} \bar \xi ^i_T \bigg] \\ 
      & \leq \E \bigg[ \int_0^T \bigg( \gamma \Big(\sum_{i=1}^N  X_t^i \Big) ^2 + \gamma \sum_{i=1}^N \big(  X_t^i \big)^2 -  X_t^0 \sum_{j=1}^N   X_t^i \bigg)  dt \\
      & \quad  -  \sum_{i=1}^N \int_0^T \Big(\gamma \sum_{j=1}^N  X_t^j + \gamma  X_t^i -  X_t^0 \Big)   \tilde X^i_t dt +  \sum_{i=1}^N c^{i}  \xi ^i_T \bigg] \\ 
      & =  \sum_{i=1}^N \mathbb E \bigg[ \int_{[0,T]} (  Y ^{i}_t + c^{i}) d  \xi ^i_t \bigg].
    \end{align*}  
    Thus, using \eqref{eq key idenity} and  \eqref{cond 2 example}, we get
    \begin{align*}
     \sum_{i=1}^N \mathbb E \bigg[ \int_{[0,T]} ( \bar Y ^{i}_t + c^{i}) d \bar \xi ^i_t \bigg] \leq \sum_{i=1}^N \mathbb E \bigg[ \int_{[0,T]} (  Y ^{i}_t + c^{i}) d \xi ^i_t \bigg]  =0,
     \end{align*} 
     which, together with \eqref{eq condition 1 bar}, in turn implies that 
     $$ 
     \text{$\int_{[0,T]} ( \bar Y ^{i}_t + c^{i}) d \bar \xi ^i_t = 0$, $\mathbb P $-a.s.}
     $$
    Finally, we can invoke Theorem \ref{thm SMP}, in order to conclude that $\bar \xi$ is a Nash equilibrium. 
\end{proof}

\smallskip 
\textbf{Acknowledgements.} 
Funded by the Deutsche Forschungsgemeinschaft (DFG, German Research Foundation) - Project-ID 317210226 - SFB 1283 
\bibliographystyle{siam}
\bibliography{main.bib}
\end{document}